\documentclass{amsart}
\usepackage{todonotes}
\usepackage{xypic}
\usepackage{graphicx, amsmath, amssymb, amsthm}
\usepackage{tikz-cd}
\usepackage{hyperref}

\newcommand{\groupsupport}[2]{The {#1} author was supported by {#2}.}

\newcommand{\NSFThree}{NSF Grant DMS-1509652}

\DeclareMathOperator{\conn}{conn}

\newcommand{\m}[1]{{\protect\underline{#1}}}

\newcommand{\mM}{\m{M}}

\newcommand{\mnu}{\m{\nu}}

\newcommand{\Sp}{\mathcal Sp}

\newcommand{\Ninfty}{N_\infty}

\newcommand{\cOrb}{\mathcal Orb}

\mathchardef\mhyphen=45

\numberwithin{equation}{section}

\newtheorem{theorem}{Theorem}[section]
\newtheorem{lemma}[theorem]{Lemma}
\newtheorem{corollary}[theorem]{Corollary}

\newtheorem{proposition}[theorem]{Proposition}

\newtheorem*{theorem*}{Theorem}
\newtheorem*{proposition*}{Proposition}

\newtheorem{MainTheorems}{Theorem}

\theoremstyle{remark}
\newtheorem{remark}[theorem]{Remark}

\theoremstyle{definition}

\newtheorem{definition}[theorem]{Definition}

\newcommand{\defemph}[1]{\textbf{#1}}

\title[New form of slices]{A new formulation of the equivariant slice filtration with applications to $C_p$-slices}

\author{Michael A.~Hill}
\address{University of California Los Angeles\\Los Angeles, CA 90025}
\email{mikehill@math.ucla.edu}
\thanks{\groupsupport{1st}{\NSFThree}}

\author{Carolyn Yarnall}
\address{California State University Dominguez Hills\\Carson, CA 90747}
\email{cyarnall@csudh.edu}

\begin{document}

\begin{abstract}
This paper provides a new way to understand the equivariant slice filtration. We give a new, readily checked condition for determining when a $G$-spectrum is slice $n$-connective. In particular, we show that a $G$-spectrum is slice greater than or equal to $n$ if and only if for all subgroups $H$, the $H$-geometric fixed points are $(n/|H|-1)$-connected. We use this to determine when smashing with a virtual representation sphere $S^V$ induces an equivalence between various slice categories. Using this, we give an explicit formula for the slices for an arbitrary $C_p$-spectrum and show how a very small number of functors determine all of the slices for $C_{p^n}$-spectra.
\end{abstract}

\keywords{equivariant stable homotopy theory, slice filtration, Mackey functor}

\subjclass[2010]{55N91, 55P91, 55Q10}
\maketitle

\section{Introduction}
The slice filtration in equivariant stable homotopy theory was introduced in the solution to the Kervaire invariant one problem \cite{HHR}. This filtration on equivariant spectra generalizes work of Dugger for $C_2$ and is analogous to Voevodsky's slice filtration in motivic homotopy (from whence the name) \cite{Dugger}, \cite{Voe:Open}. 

There is no single slice filtration for a finite group $G$. Given any sequence of collections of representation spheres, each of which is included in the next, we have an associated slice filtration. In particular, in this framework sits both the classical Postnikov filtration (for which we take for the $n$th collection all spheres of the form $S^k$ for $k$ at least $n$), the classical slice filtration in \cite{HHR}, and the regular slice filtration introduced by Ullman \cite{UllmanThesis}. Additionally, we have other versions interpolating between the Postnikov filtration and the classical or regular slice filtrations wherein we include certain regular representations. These latter filtrations are useful when analyzing the topological Andr\'e-Quillen homology of algebras over an arbitrary $\Ninfty$ operad \cite{eTAQ}, \cite{BHNinfty}. In this paper, however, we will restrict attention to the regular slice filtration. 

\begin{definition}
Let $\tau_{\geq n}$ be the localizing subcategory of genuine $G$-spectra generated by all spectra of the form
\[
G_+\wedge_H S^{k\rho_H},
\]
where $\rho_H$ is the regular representation of $H$ and where $k|H|\geq n$.

We say that a $G$-spectrum $E$ is {\defemph{slice less than $n$}} or {\defemph{slice $n$-coconnective}} if for all $H\subset G$ and $k$ such that $k|H|\geq n$, we have for all $r\geq 0$
\[
[S^{k\rho_H+r},E]^H=0.
\]

We say that a $G$-spectrum $E$ is {\defemph{slice greater than or equal to $n$}} or {\defemph{slice $n$-connective}} if $E\in\tau_{\geq n}$. 
\end{definition}

Since we will consider slice categories not only for $G$ but also for subgroups of $G$, when there is potential ambiguity, we will include superscripts to indicate the group as in $\tau_{\geq n}^G$.

In practice, it is often easier to show that something is slice $n$-coconnective, since this is a computation in homotopy groups. To show that something is slice $n$-connective, in contrast, we must either explicitly show that it is in the localizing subcategory or show that all maps from it to something slice $n$-coconnective are null. The first goal of this paper is to give a new condition for being slice $n$-connective that is similarly computable. The following is proved in Section~\ref{sec:Slices}.

\begin{MainTheorems}
A $G$-spectrum $E$ is slice $n$-connective if and only if for all subgroups $H\subset G$, the geometric fixed points $\Phi^H(E)$ are in the localizing subcategory of ordinary spectra generated by $S^{\lceil n/|H|\rceil}$.
\end{MainTheorems}

This theorem gives an immediate and readily checkable way to see when smashing with a representation sphere $S^V$ induces an equivalence between various slice categories, and this is explored in Section~\ref{sec:RepSpheres}. This generalizes the original observation from \cite[Corollary 4.24]{HHR} that smashing with the regular representation for $G$ induces an equivalence
\[
\Sigma^{\rho_G}\colon\tau_{\geq n}\xrightarrow{\cong}\tau_{\geq (n+|G|)}.
\]
The existence of more general forms of this equivalence is the original motivation for this work. Barwick asked the authors if there was such a formula, and producing it for $C_p$ led to the more general results here. We can use these equivalences to simplify the slice tower for cyclic $p$-groups, showing the following in  Section~\ref{sec:SlicesforCyclic}.

\begin{MainTheorems}
When $G=C_{p^k}$ with $p>2$, then suspension by various representations divides the slice connective categories $\tau_{\geq n}$ into $2^k$ equivalence classes.
\end{MainTheorems}

By studying the order in which these categories appear, we can also make significant progress towards determining the structure of slices.

\begin{definition}[{\cite[Section 4.2]{HHR}}]
Let $P^{n-1}(-)$ denote the nullification functor which makes $\tau_{\geq n}$ acyclic. Let $P_{\geq n}(E)$ denote the fiber of the natural map
\[
E\to P^{n-1}(E),
\]
and let $P^n_n(E)=P^n P_{\geq n}(E)$ be the {\defemph{$n$th slice}}.

The {\defemph{slice tower}} is the natural tower of functors 
\[
\dots\to P^{n+1}(-)\to P^{n}(-)\to\dots,
\]
where the natural transformation $P^{n+1}(-)\to P^{n}(-)$ arises from the obvious inclusions $\tau_{\geq (n+1)}\subset\tau_{\geq n}$.
\end{definition}

The equivalences of the localizing subcategories given by smashing with particular representation spheres produces equivalences between the slices. The first of these was used extensively in \cite{HHR}:
\begin{proposition}[{\cite[Corollary 4.25]{HHR}}]
For any finite $G$ and for any $n\in\mathbb Z$, we have a natural equivalence
\[
P^{n+|G|}_{n+|G|}\Sigma^{\rho_G}(-)\simeq \Sigma^{\rho_G} P^n_n(-).
\]
\end{proposition}

In particular, for a general group $G$, there are only $|G|$ slice functors which need to be determined: $P_n^n(-)$ for $0\leq n<|G|$. Our larger collection of equivalences simplifies this significantly for $G=C_{p^k}$. In particular, for $G=C_p$, we see that there are $3$ functors which determine all slices: the zeroth, first, and second. The first two of these were determined in \cite[Proposition 4.50]{HHR}, and a shift of the third was determined by Ullman in \cite[Corollary 8.9]{UllmanThesis}. This gives us a complete description of the slices for any $C_p$-spectrum.

\begin{MainTheorems}\label{thm:CpSlices}
When $G=C_p$ with $p>2$, then the slices of a $C_p$-spectrum $E$ are given by
\begin{align*}
& P_{mp}^{mp}(E) & \simeq & \Sigma^{m\rho} H\m{\pi}_{m\rho}(E) & \\
& P_{mp+2k+1}^{mp+2k+1}(E) & \simeq & \Sigma^{m\rho+k\lambda+1} H\mathcal P^0 \m{\pi}_{m\rho+k\lambda+1}(E) & 0\leq k\leq \tfrac{p-3}{2} \\
& P_{mp+2k+2}^{mp+2k+2}(E) & \simeq & \Sigma^{m\rho+(k+1)\lambda} H \big(EC_p\otimes\m{\pi}_{m\rho+(k+1)\lambda}(E)\big) & 0\leq k\leq \tfrac{p-3}{2},
\end{align*}
where $\mathcal P^0$ is the functor which sends a Mackey functor to the largest quotient on which the restriction map is injective and where $EC_p\otimes -$ is the functor which takes a Mackey functor to the subMackey functor generated by the underlying abelian group.
\end{MainTheorems}

\subsection*{Notation} In all that follows, let $G$ be a finite group. We will also rely heavily on a comparison with the Postnikov tower, so we fix some notation here.
\begin{definition}
Let $\tau_{\geq n}^{Post}$ denote the localizing subcategory generated by $G_+\wedge_H S^k$ for all $k\geq n$ and all $H\subset G$. We will say that $E$ is {\defemph{$n$-connective}} if it is in $\tau_{\geq n}^{Post}$.
\end{definition}

In particular, our notion of ``$n$-connective'' refers to membership in a particular localizing subcategory, and that this category is closed under smashing with any finite $G$-set.

On the algebraic side, we will work almost exclusively in the category of Mackey functors. We will denote these with underlined Roman characters. Similarly, we will denote the natural Mackey extensions of ordinary functors on abelian groups with an underline. 

\subsection*{Acknowledgements}

The authors thank Andrew Blumberg, Tyler Lawson, and Mingcong Zeng for careful reads of early drafts of this paper. The authors also heartily thank Doug Ravenel for help with the representations for a general $C_{p^n}$ in Theorem~\ref{thm:CyclicpGroupEquivalences}.

\section{A new characterization of slices}\label{sec:Slices}

\subsection{Geometric fixed points version}
One of the most surprising early results about the slice filtration is that for a class of spectra, ``geometric spectra'', the slice tower is simply a reindexed form of the Postnikov tower. 
\begin{definition}
Let $\mathcal P$ denote the family of proper subgroups of $G$. Let $E\mathcal P$ denote a universal space for $\mathcal P$, and let $\tilde{E}\mathcal P$ denote the cofiber of the obvious map $E\mathcal P_+\to S^0$.
We say that a $G$-spectrum $E$ is {\defemph{geometric}} if the natural map
\[
E\to\tilde{E}\mathcal P\wedge E,
\]
induced by smashing the inclusion $S^0\to \tilde{E}\mathcal P$ with $E$, is a $G$-equivalence.
\end{definition}

\begin{lemma}[{\cite[Theorem 6.14]{SlicePrimer}}]\label{lem:RegularPost}
If $E$ is a geometric $G$-spectrum, then the slice tower of $E$ is a reindexed form of the Postnikov tower of $E$: for all $m\in\mathbb Z$, we have
\[
P^m_m(E)\simeq\begin{cases}
\Sigma^{k\rho_G} H\m{\pi}_k(E) & m=k\cdot |G| \\
\ast & {\text{otherwise.}}
\end{cases}
\]
\end{lemma}

This gives us a shockingly powerful restatement of the slice filtration. We begin with two elementary lemmata.

\begin{lemma}\label{lem:SlicesEtildeP}
For any $G$-spectrum $E$ and for any $n\in\mathbb Z$, $\tilde{E}\mathcal P\wedge E\in\tau_{\geq n}^{G}$ if and only if $\Phi^G(E)\in\tau_{\geq n/|G|}^{Post}$.
\end{lemma}
\begin{proof}
For any $G$-spectrum $E$, $\tilde{E}\mathcal P\wedge E$ is geometric. The result is then a restatement of Lemma~\ref{lem:RegularPost}.
\end{proof}

\begin{lemma}\label{lem:SlicesEPplus}
For any $G$-spectrum $E$ and $n\in\mathbb Z$, the following are equivalent:
\begin{enumerate}
\item $E\mathcal P_+\wedge E\in\tau_{\geq n}^G$
\item for all $H\subsetneq G$, $i_H^\ast E\in\tau_{\geq n}^H$.
\end{enumerate}
\end{lemma}
\begin{proof}
Since 
\[
i_H^\ast (E\mathcal P_+\wedge E)\simeq i_H^\ast E
\]
for all proper subgroups $H$, and since for all $H\subset G$, $i_H^\ast(\tau_{\geq n}^G)\subset \tau_{\geq n}^H$, the first condition implies the second. For the other direction, we simply observe that since $i_H^\ast E\in\tau_{\geq n}^H$, we have that 
\[
G_+\wedge_H i_H^\ast E\cong G/H_+\wedge E\in\tau_{\geq n}^G.
\]
The spectrum $E\mathcal P_{+}\wedge E$ is a homotopy colimit for spectra of the form $G/H_+\wedge E$ for $H$ a proper subgroup of $G$, and since $\tau_{\geq n}^G$ is closed under homotopy colimits, we conclude it is also in $\tau_{\geq n}^G$.
\end{proof}

Putting these together, we get a very surprising new formulation of slice connectivity.

\begin{theorem}\label{thm:GeomFPVersion}
Let $n\in\mathbb Z$.
\begin{enumerate}
\item A $G$-spectrum $E$ is in $\tau_{\geq n}^G$ if and only if $i_H^\ast E\in\tau_{\geq n}^H$ for all proper subgroups $H$ and $\Phi^G(E)\in\tau_{\geq n/|G|}^{Post}$.
\item A $G$-spectrum $E$ is in $\tau_{\geq n}^G$ if and only if $\Phi^H(E)\in\tau_{\geq n/|H|}^{Post}$ for all $H\subset G$.
\end{enumerate}
\end{theorem}
\begin{proof}
The second condition follows from repeated application of the first along the lattice of subgroups.

For the first, we use isotropy separation. For any $E$, we have cofiber sequence
\[
E\mathcal P_+\wedge E\to E\to \tilde{E}\mathcal P\wedge E.
\]
For the forward direction, assume that $E\in\tau_{\geq n}^G$. This implies that for all $H\subset G$, we also have $i_H^\ast E\in\tau_{\geq n}^H$, giving the first condition. By Lemma~\ref{lem:SlicesEPplus}, $E\mathcal P_+\wedge E\in\tau_{\geq n}^G$. Since $\tau_{\geq n}^G$ is closed under taking cofibers, this implies that $\tilde{E}\mathcal P\wedge E\in\tau_{\geq n}^G$, and Lemma~\ref{lem:SlicesEtildeP} then gives the second condition.

For the reverse direction, Lemma~\ref{lem:SlicesEPplus} implies that $E\mathcal P_+\wedge E\in\tau_{\geq n}^G$, and Lemma~\ref{lem:SlicesEtildeP} implies that $\tilde{E}\mathcal P\wedge E$ is too. Since $\tau_{\geq n}^G$ is closed under extensions, this implies that $E\in\tau_{\geq n}$.
\end{proof}


\subsection{Connectivity version}
When $X$ is a $(-1)$-connected spectrum, then there is yet another version of the slice tower, this one based on connectivity of the fixed points, rather than the geometric fixed points. In his thesis, Ullman uses a decomposition of cells from $\tau_{\geq n}$ for negative $n$ along with an analogue of Brown-Comenetz duality to show that slice $n$-connectivity, for positive $n$, is equivalent to a statement about equivariant connectivity. The ultimate results in this section are equivalent to \cite[Theorem 8.10]{UllmanThesis}. Here, we present an alternative proof which makes use of the isotropy separation sequence.

Recall that the notion of connectivity for $G$-spaces and spectra involves the fixed points for all subgroups. In particular, for any $G$-spectrum $E$, the connectivity of $E$ is a function from the isomorphism classes of orbits of $G$ to the extended integers $\mathbb Z\cup\{\pm\infty\}$:
\[
\conn(E)(G/H):=\conn(E^H).
\]

\begin{definition}
If $n\geq 0$, let 
\[
\mnu_n\colon \pi_0 \cOrb_G\to\mathbb N\cup \{\infty\}
\]
be defined by
\[
\mnu_n(G/H)=\left\lceil \frac{n}{|H|}\right\rceil.
\]
\end{definition}

\begin{definition}
Let $\tau_{\geq \mnu_n}$ denote the localizing subcategory of $\Sp^G$ generated by all $G$-spectra $E$ with
\[
\conn(E)(G/H)\geq\mnu_n(G/H)
\]
for all $H\subset G$.
\end{definition}

These are closely related to the various slice positive categories. 

\begin{proposition}[{\cite[Lemma 4.38]{HHR}}]
If $k\cdot |H|\geq n$, then
\[
G_+\wedge_H S^{k\rho_H}\in\tau_{\geq\mnu_n}.
\]
\end{proposition}
\begin{proof}
We must show that for all $K\subset G$, the $K$-fixed points of $G_+\wedge_H S^{k\rho_H}$ has connectivity at least $\mnu_n(G/K)$. The double coset formula gives us an equivalence
\[
i_K^\ast \Big(G_+\wedge_H S^{k\rho_H}\Big)\simeq \bigvee_{g\in K\backslash G / H} K_+\wedge_{K\cap gHg^{-1}} S^{i_K^\ast k\rho_{gHg^{-1}}}.
\]
The $K$-fixed points of the summand corresponding to $g$ have connectivity
\begin{multline*}
k\cdot [H\colon H\cap g^{-1} Kg]\geq \left\lceil \frac{n}{|H|}\right\rceil [H\colon H\cap g^{-1} K g] \geq \left\lceil \frac{n\cdot [H\colon H\cap g^{-1} Kg]}{|H|}\right\rceil \\ \geq \left\lceil \frac{n}{|H\cap g^{-1} Kg|}\right\rceil=\left\lceil\frac{n}{|gHg^{-1}\cap K|}\right\rceil\geq \mnu_n(G/K).
\end{multline*}
\end{proof}

\begin{corollary}[{\cite[Proposition 4.40]{HHR}}]
For any $n\geq 0$, 
\[
\tau_{\geq n}\subset \tau_{\geq \mnu_n}.
\]
\end{corollary}

It is somewhat surprising, but the converse is also true. This gives an alternative characterization of being slice $\geq n$. 
\begin{theorem}\label{thm:Slices}
For $n\geq 0$, we have
\[
\tau_{\geq \mnu_n}\subset\tau_{\geq n}.
\]
\end{theorem}
\begin{proof}
We prove this by induction on the order of the group. For the trivial group, this is by definition, since the non-equivariant slice filtration is the same as the Postnikov filtration. Now assume that for all proper subgroups $H$, 
\[
\tau_{\geq\mnu_n}^H\subset\tau_{\geq n}^H.
\]
In particular, this implies that if $X$ is any $H$ spectrum with connectivity at least $\mnu_n$, then the induced spectrum $G_+\wedge_H X$ is in $\tau_{\geq n}$.

Consider $X$ in $\tau_{\geq\mnu_n}$. We argue via the isotropy separation sequence:
\[
E\mathcal P_+\wedge X\to X\to \tilde{E}\mathcal P\wedge X.
\]
The spectrum $E\mathcal P_+\wedge X$ is a homotopy colimit of a diagram built out of the restrictions of $X$ to proper subgroups. By the induction hypothesis, this is then in $\tau_{\geq n}$. Since $\tau_{\geq n}$ is a localizing subcategory, it therefore suffices to show that $\tilde{E}\mathcal P\wedge X$ is in $\tau_{\geq n}$. However, by assumption, the bottom homotopy group of $X$, and hence of $\tilde{E}\mathcal P\wedge X$ is in dimension at least $\mnu_n(G/G)$. This means that all slices of $\tilde{E}\mathcal P\wedge X$ are in dimensions at least 
\[
\mnu_n(G/G)\cdot |G|\geq n,
\]
and hence $\tilde{E}\mathcal P\wedge X$ is in $\tau_{\geq n}$.
\end{proof}


\section{Representation Spheres}\label{sec:RepSpheres}
One of the nicest features of this new characterization of slice greater than or equal to $n$ is that it makes showing that spectra, like representation spheres, are in a particular slice category absurdly easy. Even more exciting, it allows one to trivially check that smashing with a representation sphere moves us between slice categories the way we might like. This was a major undertaking in \cite{SlicePrimer} and \cite{Yarnall}, as it often required judicious choices of ambient representations and various downward inductions on cells.

\begin{definition}
If $V$ is a representation of $G$, then let $\m{\dim}_V$ be the function on isomorphism classes of orbits given by
\[
\m{\dim}_V(G/H)=\dim V^H.
\]
\end{definition}

\begin{theorem}
The representation sphere $S^V$ is in $\tau_{\geq n}$ if and only if for all $H\subset G$.
\[
\dim V^H\geq n/|H|.
\]
\end{theorem}
\begin{proof}
For all $H\subset G$, we have a natural isomorphism $\Phi^H(S^V)\cong S^{V^H}$. In particular
\[
\Phi^H(S^V)\cong S^{V^H}\in\tau_{\geq \m{\dim}_V(G/H)}^{Post}\subset\tau_{\geq n/|H|}^{Post}.
\]
The result follows from Theorem~\ref{thm:GeomFPVersion}.
\end{proof}

This gives a similarly vast generalization of how certain suspensions change slice connectivity and moreover why other suspensions do not. We use the following standard result, omitting the proof.

\begin{proposition}
If $E\in\tau_{\geq n}^{Post}$ and $E'\in\tau_{\geq m}^{Post}$, then $E\wedge E'\in\tau_{\geq n+m}^{Post}$.
\end{proposition}

\begin{theorem}\label{thm:Smashing}
Let $V$ be a virtual representation of $G$. If $\m{\dim}_V+\mnu_n\geq\mnu_{n+k}$, then smashing with $S^V$ induces a map
\[
\Sigma^V\colon\tau_{\geq n}\to\tau_{\geq n+k}.
\]

Moreover $\m{\dim}_V+\mnu_n=\mnu_{n+k}$ if and only if this map is an equivalence.
\end{theorem}
\begin{proof}
The first part follows immediately from the assumptions and from the previous proposition. The second follows from the observation that for this result, we never needed that $V$ be an actual representation, and hence smashing with $S^{-V}$ provides the inverse.
\end{proof}

We deduce from this a very surprising collection of auto-equivalences of the categories $\tau_{\geq n}$ for all $n$.

\begin{corollary}\label{cor:SliceEquivalences}
If $V$ is a virtual representation of $G$ such that for all $H\subset G$, 
\(
\m{\dim}_V(G/H)=0,
\)
then smashing with $S^V$ induces an auto-equivalence:
\[
\Sigma^V\colon\tau_{\geq n}\xrightarrow{\cong} \tau_{\geq n}
\] 
for all $n$.
\end{corollary}
These auto-equivalences are ubiquitous. In Proposition~\ref{prop:CyclicAutoEquivalences} below, we show that there are such auto-equivalences even for cyclic $p$-groups of order at least $5$.

Applying this to the slices themselves, we deduce an even more surprising corollary.
\begin{corollary}
If $V$ is a virtual representation of $G$ such that for all $H\subset G$, 
\(
\m{\dim}_V(G/H)=0, 
\)
then smashing with $S^V$ commutes with the formation of slices: for all $n$ and for all $G$-spectra $E$, we have
\[
\Sigma^V P_n^n(E)\simeq P_n^n (\Sigma^V E).
\]
\end{corollary}

Several non-inclusion results also follow immediately from Theorem~\ref{thm:Smashing}: if we
do not have the desired inequality, then we do not have such a nice
embedding. An example is given by the reduced regular representation
spheres. Classically,
\[
S^{2\bar{\rho}}\not\in\tau_{\geq 2(|G|-1)}
\]
even though $S^{\bar{\rho}}$ is in $\tau_{\geq |G|-1}$. Here we see why: for all $G$, $\mnu_{2|G|-2}(G/G)>0$, while the connectivity of $S^{\bar{\rho}}$ has $\m{\dim}_{\bar{\rho}}(G/G)=0$.

\section{Slices for Cyclic $p$-groups}\label{sec:SlicesforCyclic}

Specializing to cyclic $p$-groups with $p$ odd, we have an interesting series of results. Choose a $p^n$th root of unity $\zeta$ and let $\lambda(k)$ denote the representation of $C_{p^n}$ sending a generator $\gamma$ to $\zeta^k$. The representations $\lambda(k)$ for $1\leq k\leq (p-1)/2$ are representatives for the isomorphism classes of non-trivial real representations of $C_{p^n}$. Work of Kawakubo shows that the representation spheres are all integrally inequivalent \cite{Kawakubo80}. Even so, the exact choices are immaterial due to the following observation.

\begin{proposition}\label{prop:CyclicAutoEquivalences}
If the $p$-adic valuations of $k$ and $k'$ agree, then smashing with 
\[
S^{\lambda(k)-\lambda(k')}
\]
induces an automorphism of the slice categories $\tau_{\geq m}$ for all $m$.
\end{proposition}
\begin{proof}
For all subgroups of $C_{p^n}$, the fixed points of $\lambda(k)$ and $\lambda(k')$ agree. The result follows from Corollary~\ref{cor:SliceEquivalences}.
\end{proof}

Since any two $k$ with equal $p$-adic valuations yield naturally equivalent functors, we will  let $\lambda_k=\lambda(p^k)$. Finally, let $\lambda=\lambda_0$.

\begin{definition}\label{def:Vj}
Let $k\geq 1$ and $0 \leq j \leq k-1$. Define the
$C_{p^k}$-representation $V_j$ as follows

\[
V_j = \bigoplus_{i=0}^j \left(p^i - \lfloor p^{i-1}\rfloor\right) \lambda_{j-i}
\]

\end{definition}

\begin{theorem}\label{thm:CyclicpGroupEquivalences}
For $G = C_{p^k}$ and $0 \leq j \leq k-1$, smashing with $S^{V_j}$ induces equivalences
\[
\Sigma^{V_j}\colon\tau_{\geq n}\xrightarrow{\cong} \tau_{n+2p^j}
\]
for all $n \geq 0$, $n \equiv m \mbox{ (mod }p^{j+1})$ where $1\leq m \leq p^{j+1}-2p^j$.
\end{theorem}

\begin{proof}
Let $n = m + \ell p^{j+1}$ where $1 \leq m \leq p^{j+1}-2p^j$ and
$\ell \geq 0$. By \autoref{thm:Smashing} we need to show that 
\[
\dim\left(V_j^{C_{p^d}}\right) + \left\lceil \frac{m+\ell
    p^{j+1}}{p^d} \right\rceil = \left\lceil \frac{m +\ell p^{j+1} +
    2p^j}{p^d} \right\rceil
\]
for all $0 \leq j \leq k-1$ and $0 \leq d \leq k$. 
From Definition~\ref{def:Vj} we can determine that
\[
\dim\left(V_j^{C_{p^d}}\right) = \begin{cases} 2p^{j-d} & d \leq j\\
0 & d > j \end{cases}
\]
Thus, when $d \leq j$, the result is immediate so we need only
consider $d > j$ or equivalently $d \geq j+1$. 

First, we may write $\ell p^{j+1} = qp^d + r$ where $q\geq 0$ and
$0\leq r \leq p^d - p^{j+1}$ since $p^{j+1}|r$. Then, 
\[
\left\lceil \frac{m+\ell
    p^{j+1}}{p^d} \right\rceil = \left\lceil \frac{m+qp^d + r}{p^d}
\right\rceil= \left\lceil \frac{m+r}{p^d} \right\rceil + q
\]
and
\[
\left\lceil \frac{m+\ell
    p^{j+1}+2p^j}{p^d} \right\rceil =\left\lceil \frac{m+r + 2p^j}{p^d} \right\rceil + q.
\]
So now we must show
\[
\left\lceil \frac{m+r}{p^d} \right\rceil = \left\lceil \frac{m+r +
    2p^j}{p^d} \right\rceil .
\]

Since $m \leq p^{j+1} - 2p^j$ and $r \leq p^d - p^{j+1}$ we have
\[
m + r \leq m +r + 2p^j \leq p^{j+1} + p^d - p^{j+1} = p^d.
\]
 Then since $m, r \geq 0$, we have 
\[
\left\lceil \frac{m+r}{p^d} \right\rceil = 1 = \left\lceil \frac{m+r +
    2p^j}{p^d} \right\rceil 
\]
which completes the proof.
\end{proof}

\begin{corollary}
The equivalences given in \autoref{thm:CyclicpGroupEquivalences} along
with the equivalences induced by smashing with $S^{\rho_{p^k}}$ sort
the categories $\tau_{\geq n}$ into $2^k$ equivalence classes defined
by 
\[
n =\sum_{i=0}^{k-1} a_i p^i \mbox{ where } a_i = \begin{cases} 1,2 & i = 0\\
  0,1 &1 \leq i \leq k-1 \end{cases}
\]

\end{corollary}

We specialize now to the case of $C_p$, where we can describe more explicitly the slices for any spectrum.

\begin{theorem}\label{thm:CpSliceCategories}
For $G=C_p$, smashing with $S^{\lambda}$ induces equivalences
\[
\Sigma^{\lambda}\colon\tau_{\geq (2j-1)}\xrightarrow{\cong} \tau_{\geq (2j+1)}
\]
for $1\leq 2j-1\leq p-2$ and
\[
\Sigma^{\lambda}\colon\tau_{\geq 2j}\xrightarrow{\cong}\tau_{\geq 2j+2}
\]
for $2\leq 2j\leq p-3$.
\end{theorem}

This lets us rewrite any slice in terms of basic operations.

\begin{corollary}\label{cor:CpSlicesOne}
For a $C_p$-spectrum $E$, we have equivalences for any $m\in\mathbb Z$
\begin{align*}
& P_{mp}^{mp}(E) & \simeq & \Sigma^{m\rho} H\m{\pi}_{m\rho}(E) & \\
& P_{mp+2k+1}^{mp+2k+1}(E) & \simeq & \Sigma^{m\rho+k\lambda+1} H\mathcal P^0 \m{\pi}_{m\rho+k\lambda+1}(E) & 0\leq k\leq \tfrac{p-3}{2} \\
& P_{mp+2k+2}^{mp+2k+2}(E) & \simeq & \Sigma^{m\rho+k\lambda} P_2^2\big(\Sigma^{-m\rho-k\lambda}(E)\big) & 0\leq k\leq \tfrac{p-3}{2},
\end{align*}
where $\mathcal P^0$ is the functor that takes a Mackey functor to its largest quotient in which the restriction map is an injection.
\end{corollary}
\begin{proof}
Smashing with the regular representation always moves between the appropriate slice categories, hence it suffices to consider the case $m=0$. In this case, Theorem~\ref{thm:CpSliceCategories} shows that smashing with $\Sigma^{k\lambda}$ induces an equivalence between $1$-slices and $(2k+1)$-slices and between $2$-slices and $(2k+2)$-slices for $0\leq 2k\leq p-3$. The result follows from the determination of the $0$- and $1$-slices.
\end{proof}

Work of Ullman allows us to complete the characterization of $C_p$-slices. UIlman describes the $(-1)$-slice for any finite group $G$, and we recall the results here.

\begin{definition}
If $\mM$ is a $G$-Mackey functor, then let $EG\otimes \mM$ be the sub-Mackey functor generated by $\mM(G/e)$.
\end{definition}
\begin{remark}
The functor $EG\otimes \mM$ is so named because it is $\m{\pi}_0$ of the spectrum $EG_+\wedge H\mM$.
\end{remark}

\begin{theorem}[{\cite[Corollary 8.9, Section I.8]{UllmanThesis}}]\label{thm:UllmansPiMinusOne}
For any finite group $G$ and for any $G$-spectrum $E$, we have
\[
P_{-1}^{-1} E\simeq\Sigma^{-1} H \big(EG\otimes \m{\pi}_{-1}E\big).
\]
\end{theorem}

This gives a simple formula for the $2$-slices for $C_p$.

\begin{corollary}
For any $C_p$-spectrum $E$, we have 
\[
P_2^2 E\simeq \Sigma^{\lambda} H\big(EC_p\otimes \m{\pi}_{\lambda} E\big).
\]
\end{corollary}
\begin{proof}
Corollary~\ref{cor:CpSlicesOne} with $m=-1$ and $k=(p-3)/2$ shows that
\[
P_{-1}^{-1} E\simeq \Sigma^{-\rho+(p-3)/2\lambda} P_2^2 \big(\Sigma^{\rho-(p-3)/2\lambda} E\big)
\]
or equivalently
\[
P_2^2 E\simeq \Sigma^{1+\lambda} P_{-1}^{-1}\big(\Sigma^{-1-\lambda} E\big).
\]
Theorem~\ref{thm:UllmansPiMinusOne} then gives the result.
\end{proof}

Summarizing, we have the following determination of all of the $C_p$-slices, completing the proof of Theorem~\ref{thm:CpSlices} from the introduction.

\begin{theorem}
Let $p$ be odd. For any $C_p$-spectrum $E$, we have equivalences for any $m\in\mathbb Z$
\begin{align*}
& P_{mp}^{mp}(E) & \simeq & \Sigma^{m\rho} H\m{\pi}_{m\rho}(E) & \\
& P_{mp+2k+1}^{mp+2k+1}(E) & \simeq & \Sigma^{m\rho+k\lambda+1} H\mathcal P^0 \m{\pi}_{m\rho+k\lambda+1}(E) & 0\leq k\leq \tfrac{p-3}{2} \\
& P_{mp+2k+2}^{mp+2k+2}(E) & \simeq & \Sigma^{m\rho+(k+1)\lambda} H \big(EC_p\otimes\m{\pi}_{m\rho+(k+1)\lambda}(E)\big) & 0\leq k\leq \tfrac{p-3}{2}.
\end{align*}
\end{theorem}

\bibliographystyle{plain}

\bibliography{Slices}

\end{document}